\documentclass[a4paper, reqno, twoside, 10pt]{amsart}

\usepackage{amsmath}
\usepackage{amssymb}
\usepackage{graphicx}
\usepackage{enumerate}
\usepackage{comment}
\usepackage{hyperref}
\usepackage{xspace}
\usepackage{mathtools}
\usepackage{adjustbox}
\usepackage{enumitem}

\usepackage{tikz}
\usetikzlibrary{arrows,patterns,decorations.pathmorphing,backgrounds,positioning,fit,petri,decorations.pathreplacing}

\tikzset{bullet/.style={
shape = circle,fill = black, inner sep = 0pt, outer sep = 0pt, minimum size = 0.35em, line width = 0pt, draw=black!100}}

\tikzset{circle/.style={
shape = circle,fill = none, inner sep = 0pt, outer sep = 0pt, minimum size = 0.35em, line width = 1pt, draw=black!100}}

\tikzset{rectangle/.style={
shape = rectangle,fill = white, inner sep = 0pt, outer sep = 0pt, minimum size = 0.35em, line width = 0pt, draw=black!100}}

\tikzset{empty/.style={
shape = circle,fill = white, inner sep = 0pt, outer sep = 0pt, minimum size = 0.35em, line width = 0pt, draw=white!100}}

\tikzset{xmark/.style={
shape = x,fill = white, inner sep = 0pt, outer sep = 0pt, minimum size = 0em, line width = 0pt, draw=white!100}}

\tikzset{longrectangle/.style={
inner sep = 1em,
rectangle,
minimum size=1em,
very thick,
draw=black!100,
}}

\tikzset{label distance=-0.15em}
\tikzset{font=\small}

\usepackage{tikz-cd}
\tikzcdset{every label/.append style = {font = \normalsize}}
\usepackage{tkz-graph}

\hypersetup{%
hidelinks,%
pdftitle={Normal complex surface singularities with rational homology disk smoothings},%
pdfauthor={M\'arton Beke, Heesang Park, Dongsoo Shin, Andr\'as I. Stipsicz},%
pdfkeywords={complex surface singularity, Milnor fiber,
rational homology disk, weighted homogeneous singularity},%
}

\newtheorem{theorem}{Theorem}[section]
\newtheorem*{theorem*}{Theorem}
\newtheorem{lemma}[theorem]{Lemma}
\newtheorem*{lemma*}{Lemma}
\newtheorem{proposition}[theorem]{Proposition}
\newtheorem{corollary}[theorem]{Corollary}
\newtheorem*{corollary*}{Corollary}
\newtheorem*{maintheorem*}{Main Theorem}

\theoremstyle{definition}

\newtheorem*{remark*}{Remark}

\numberwithin{equation}{section}

\DeclareMathOperator{\Supp}{Supp}

\DeclareMathOperator{\im}{im}

\DeclareMathOperator{\coker}{coker}

\DeclareMathOperator{\val}{val}

\newcommand{\QHD}{\(\mathbb{Q}\)HD\xspace}
\newcommand{\QHDS}{\(\mathbb{Q}\)HD smoothing\xspace}

\begin{document}

\title[Singularities with QHD smoothings]{Normal complex surface singularities with rational homology disk smoothings}

\author[M. Beke]{M\'arton Beke}

\address{R\'enyi Institute of Mathematics, Budapest and \\ University of Technology and Economics, Budapest, Hungary}
\email{bekem@renyi.hu}

\author[H. Park]{Heesang Park}

\address{Department of Mathematics, Konkuk University, Seoul, Republic of Korea}

\email{HeesangPark@konkuk.ac.kr}

\author[D. Shin]{Dongsoo Shin}

\address{Department of Mathematics, Chungnam National University, Daejeon, Republic of Korea}

\email{dsshin@cnu.ac.kr}

\author[A. I. Stipsicz]{Andr\'{a}s I. Stipsicz}

\address{R\'enyi Institute of Mathematics, Budapest, Hungary}

\email{stipsicz.andras@renyi.hu}




\begin{abstract}
We prove that a normal surface singularity whose minimal good resolution graph has at least two nodes does not admit a rational homology disk smoothing. Together with the earlier graph classifications and the analytic classifications of Fowler and Wahl, this gives a complete classification of normal surface singularities admitting rational homology disk smoothings and proves Wahl's conjecture. The same argument also proves Wahl's Rational Conjecture.
\end{abstract}

\maketitle

\section{Introduction}
\label{sec:introduction}

Let \((X,0)\) be a normal complex surface singularity. A smoothing
of \((X,0)\) is a flat deformation \(\pi \colon (\mathcal{X},0) \to (\Delta,0)\) 
over a disc with central fiber \((X,0)\) and smooth nearby fibers. Its
Milnor fiber \(M\) is a compact 4-manifold bounded by the link of
\((X,0)\), and its Milnor number is \(\mu=b_2(M)\). A \emph{rational
homology disk smoothing} (\QHDS for short) is a smoothing for which
\(M\) has the rational homology of a 4-disk. This is equivalent to
\(\mu=0\), since \(M\) has the homotopy type of a two-dimensional CW
complex and \(b_1(M)=0\) by
Greuel--Steenbrink~\cite[Theorem~2]{GreuelSteenbrink1983}.

The classification of normal surface singularities admitting \QHD
smoothings is a longstanding problem, dating back to examples
constructed in the early 1980s. Among cyclic quotient surface
singularities, precisely those of type \(\frac{1}{p^2}(1,pq-1)\), with
\(p>q>0\) relatively prime, admit a \QHDS; see
Looijenga--Wahl~\cite[(5.10)]{LooijengaWahl1986} and
Wahl~\cite[(5.9.1)]{Wahl1981}. 
In the symplectic category, Lisca~\cite{Lisca08} classified the symplectic fillings of lens spaces with their canonical contact structures, showing in particular that they admit symplectic rational homology disk fillings precisely when they are links of cyclic quotient singularities of the above type.
These smoothings occur in the KSBA
compactification of moduli spaces of surfaces of general type. For
example, Y.~Lee and J.~Park~\cite{Lee-Park-K^2=2} constructed such
surfaces by smoothing a surface with cyclic quotient singularities of
the above type; see
also~\cite{Keum-Lee-Park,PPS-K3,PPS-K4,PPS-pg1,PPS-H1Z4,PSU}. These
constructions were motivated by the \emph{rational blow-down}
construction of Fintushel--Stern~\cite{Fintushel-Stern-1997}, and by
its generalization by J. Park~\cite{JPark-1997}. In the rational
blow-down construction, the tubular neighborhood of a configuration of
surfaces in a 4-manifold, whose intersection pattern is given by the
resolution graph, is replaced by the corresponding rational homology
disk Milnor fiber. Rational blow-down has also played a central role
in constructions of exotic smooth structures on many 4-manifolds,
cf. for example~\cite{Fintushel-Stern-1997,JPark-2005,PSSz,SSz}.

A singularity admitting a \QHDS is rational by Durfee's
formula~\cite{Durfee1978}; its resolution graph is therefore a tree of
rational curves.  In addition to the cyclic quotient ones, further
examples were constructed by Wahl~\cite{Wahl1981}; a list of these was
already known, as noted by de Jong--van
Straten~\cite[p.~505]{deJong-VanStraten-1998}.

Using smooth topological methods, Stipsicz--Szab\'o--Wahl~\cite{StipsiczSzaboWahl2008} established strong necessary combinatorial conditions on the resolution graphs of singularities admitting a \QHDS. Bhupal--Stipsicz~\cite[Theorems~1.4 and 1.6]{BhupalStipsicz2011} then classified the nonlinear star-shaped resolution graphs admitting a \QHDS. These are precisely the graphs in Figures~\ref{figure:fig1} and \ref{figure:fig2}; see also Beke~\cite{Bek26}. Wahl~\cite{Wahl2011QHD} conjectured that the singularities admitting such smoothings are exactly the known weighted homogeneous examples.

We prove Wahl's conjecture by excluding resolution graphs with at least two nodes, where a \emph{node} is a vertex of valency at least three.

\begin{theorem}\label{thm:main}
A normal complex surface singularity whose minimal good resolution graph has at least two nodes does not admit a rational homology disk smoothing.
\end{theorem}

Together with the cyclic quotient surface singularity case and the result of Bhupal--Stipsicz~\cite{BhupalStipsicz2011}, Theorem~\ref{thm:main} gives the following classification of resolution graphs.

\begin{corollary}\label{corollary:ClassificationOfGraphs}
The minimal good resolution graph of a normal surface singularity admitting a \QHDS is either the linear graph of a cyclic quotient singularity of type \(\frac{1}{p^2}(1,pq-1)\), with \(p>q>0\) relatively prime, or one of the graphs in Figures~\ref{figure:fig1} and \ref{figure:fig2}.
\end{corollary}

\begin{figure}[t]
\centering
\includegraphics[width=\textwidth]{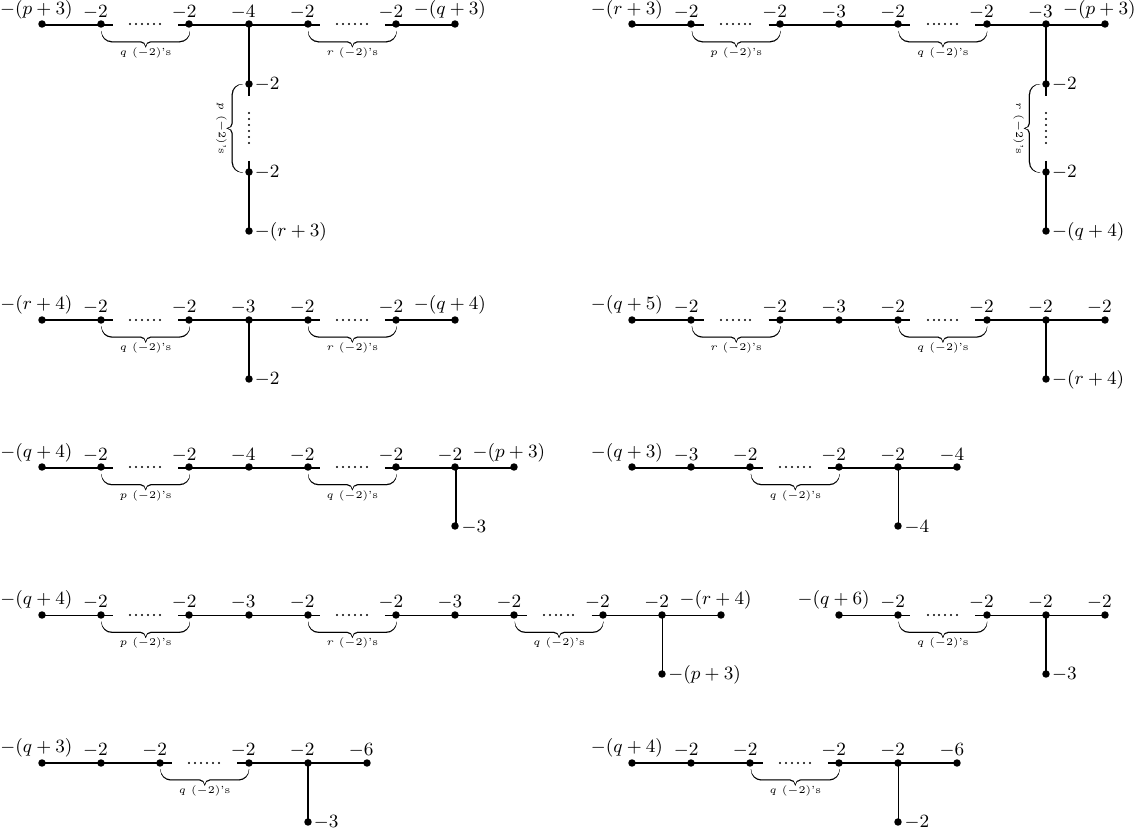}
\caption{Star-shaped graphs with one node of valency three corresponding to singularities with a \QHDS. We assume that \(p,q,r \ge 0\).}
\label{figure:fig1}
\end{figure}

\begin{figure}[t]
\centering
\includegraphics[scale=0.75]{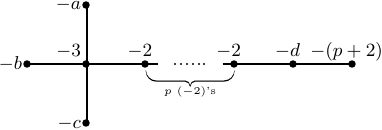}
\caption{Star-shaped graphs with one node of valency four corresponding to singularities with a \QHDS. The quadruple \((a,b,c;d)\) is one of \(\{(3,3,3;4),(2,4,4;3),(2,3,6;2)\}\), and \(p \ge 0\).}
\label{figure:fig2}
\end{figure}

This classification concerns complex smoothings, not symplectic fillings: indeed, some links whose resolution graphs are excluded from the smoothing classification nevertheless admit Stein \QHD fillings, as shown in Beke--Plamenevskaya--Starkston~\cite{BPS}.

The cyclic quotient graphs and the graphs in Figure~\ref{figure:fig1}
are taut, meaning that each determines a unique analytic type. For a
graph in Figure~\ref{figure:fig2}, the analytic type is determined by
the four intersection points on the central curve  with their
corresponding arms. All these analytic types are weighted homogeneous;
see Laufer~\cite{Laufer1973}, Pinkham~\cite{Pinkham1977}, and
Wahl~\cite[Introduction and proof of Corollary~8.2]{Wahl2011QHD}.

To specify the cross ratio for a graph of Figure~\ref{figure:fig2}, let \(E_0\) be the central component of the exceptional divisor. Denote its intersections with the single-vertex arms of weights \(-a,-b,-c\) by \(x_a,x_b,x_c\), respectively, and its intersection with the remaining arm by \(x_d\). Let \(\varphi \colon E_0 \to \mathbb{P}^1\) be the unique isomorphism satisfying \(\varphi(x_a)=0\), \(\varphi(x_b)=1\), and \(\varphi(x_c)=\infty\). Define the \emph{normalized cross ratio} by
\begin{equation}\label{eq:cross-ratio}
 \lambda=\varphi(x_d) \in \mathbb{C} \setminus \{0,1\}.
\end{equation}
When comparing analytic types, only permutations of these points induced by automorphisms of the weighted graph are allowed.

Fowler's classification of model surfaces~\cite[Theorem~3.3.1 and Section~3.5]{Fowler2013}, summarized in Wahl~\cite[Corollary~1.5]{Wahl2021QHD}, determines the cross ratios for which a \QHDS exists. These results, together with Theorem~\ref{thm:main}, give the following complete analytic classification.

\begin{corollary}\label{cor:PreciseMain}
The normal complex surface singularities admitting a \QHDS are precisely the following:

\begin{itemize}
\item the cyclic quotient singularities of type \(\frac{1}{p^2}(1,pq-1)\), with \(p>q>0\) relatively prime;

\item the weighted homogeneous singularities corresponding to the taut graphs of Figure~\ref{figure:fig1}; and

\item the weighted homogeneous singularities with resolution graphs of Figure~\ref{figure:fig2} whose normalized cross ratio is \(\lambda=e^{\pm\pi i/3}\) for \((3,3,3;4)\), \(\lambda=2\) for \((2,4,4;3)\), and \(\lambda=9\) for \((2,3,6;2)\).
\end{itemize}
\end{corollary}

The results on smoothing components in Wahl~\cite[Corollary~8.2]{Wahl2011QHD} and \(\mathbb{Q}\)-Gorenstein smoothings in Wahl~\cite[Theorem~3.4]{Wahl-2011-log}, together with Theorem~\ref{thm:main}, also give the following consequence.

\begin{corollary}\label{cor:qhd-properties}
Every rational homology disk smoothing of a normal surface singularity is \(\mathbb{Q}\)-Gorenstein and occurs over a one-dimensional smoothing component.
\end{corollary}

The cohomological comparison used to prove Theorem~\ref{thm:main} also proves Wahl's Rational Conjecture~\cite[Section~1]{Wahl2017}. For a minimal good resolution \(f \colon (V,E) \to (X,0)\), put
\[
 \mathcal{S}=\Theta_V(-\log{E}),
 \quad
 \mathcal{L}=\mathcal{O}_V(-(K_V+E)).
\]

\begin{theorem}[Wahl's Rational Conjecture]\label{prop:rational-conjecture}
Let \((V,E) \to (X,0)\) be the minimal good resolution of a rational surface singularity which is not a rational double point. Then
\[
 h^1(V,\mathcal{S}) \le h^1(V,\mathcal{L}),
\]
with equality if and only if \((X,0)\) is quasi-homogeneous.
\end{theorem}

\begin{remark*}
Park--Shin~\cite{ParkShin2026} prove that the Wahl defect \(\Delta(X,0) := h^1(V,\mathcal{O}_V)-h^1(V,\mathcal{S})+h^1(V,\mathcal{L})\) of a non-Gorenstein \(\mathbb{Q}\)-Gorenstein normal surface singularity is nonnegative, with equality precisely in the weighted homogeneous case. As an application, they show that every normal surface singularity admitting a rational homology disk smoothing is weighted homogeneous. Together with the known classification in the weighted homogeneous case, this also gives Corollary~\ref{cor:PreciseMain}.
\end{remark*}

\subsection{Outline of the proof}

The obstruction to a \QHDS comes from the dimension of its smoothing component. Let \(f \colon (V,E) \to (X,0)\) be the minimal good resolution of a rational surface singularity, with resolution graph \(\Gamma\), reduced exceptional divisor \(E=\sum_{i=0}^{n}E_i\), and \(E_i^2=-d_i\). By Wahl~\cite[Theorem~8.1]{Wahl2011QHD}, every \QHDS component has dimension 
\begin{equation}\label{eq:dimension}
 h^1(V,\mathcal{S})+\sum_{i=0}^{n}(d_i-3).
\end{equation}

The graph restrictions of Stipsicz--Szab\'o--Wahl~\cite{StipsiczSzaboWahl2008}
imply that if a singularity admits a \QHD smoothing and its resolution
graph has at least two nodes, then the graph belongs to one of the
classes \(\mathcal{A}\), \(\mathcal{B}\), or \(\mathcal{C}\).
Since every smoothing component has positive dimension, to prove
Theorem~\ref{thm:main} it suffices to establish the inequality
\begin{equation*}
 h^1(V,\mathcal{S})+\sum_{i=0}^{n}(d_i-3) \le 0
\end{equation*}
for every rational singularity whose resolution graph belongs to one
of these classes and has at least two nodes. We prove this inequality
in Theorem~\ref{thm:main-technical}.

We first compare \(h^1(V,\mathcal{S})\) with
\(h^1(V,\mathcal{L})\). Let \(E'\) be the reduced sum of the
exceptional components of valency at least two, and let \(r\) be the
number of leaves of \(\Gamma\). Wahl's restriction calculation
~\cite{Wahl2017} gives the identity
\[
 h^1(V,\mathcal{S})=r-4+h^1(V,\mathcal{S}(-E')).
\]
In Sections~\ref{sec:computation-sequences} and
\ref{sec:cohomological-comparison}, we use Giraud's computation sequence from \(E'\) to \([\mathcal{L}]\) to prove Proposition~\ref{prop:cohomological-comparison}:
\[
 h^1(V,\mathcal{S}(-E')) \le h^1(V,\mathcal{L}(-E')).
\]
Giraud's cohomology formula \cite{Giraud1982} also gives the identity 
\[
 h^1(V,\mathcal{L}(-E'))=h^1(V,\mathcal{L})-(r-3).
\]
Hence \(h^1(V,\mathcal{S})\le h^1(V,\mathcal{L})-1\).

It remains to bound \(h^1(V,\mathcal{L})\) in terms of the resolution graph. The construction of the dual graph \(\Gamma\) in the class \(\mathcal{A}\), \(\mathcal{B}\), or \(\mathcal{C}\) determines a nonnegative integer \(m\), whose precise definition is given in Subsection~\ref{subsec:graph-constructions}. In Proposition~\ref{prop:lattice-estimate}, we prove that
\(
 h^1(V,\mathcal{L}) \le m.
\)
The proof expresses this cohomology dimension as the value of a quadratic function on the exceptional lattice and bounds it by following the graph construction. The same construction gives~\eqref{eq:sum<=1-m}:
\(
 \sum_{i=0}^{n}(d_i-3)=1-m.
\)
Hence
\[
 h^1(V,\mathcal{S})+\sum_{i=0}^{n}(d_i-3) 
 \le h^1(V,\mathcal{L})-1+(1-m)
 =h^1(V,\mathcal{L})-m
 \le 0.
\]
Wahl's Rational Conjecture is proved in Subsection~\ref{subsec:rational-conjecture} as a consequence of the cohomological comparison. Throughout the paper we work over the field of complex numbers.

\subsection{Acknowledgements}

The authors thank J.~Wahl for his careful reading and valuable comments, and for pointing out an error in the first draft of this paper. They also thank A.~Hof for helpful discussions. MB is partially supported by the Doctoral Excellence Fellowship Programme (DCEP) funded by the National Research Development and Innovation Fund of the Ministry of Culture and Innovation and the Budapest University of Technology and Economics, and by the ERC Advanced Grant KnotSurf4d. HP was supported by the Basic Science Research Program through the National Research Foundation of Korea (RS-2026-25590534). DS was supported by the Basic Science Research Program through the National Research Foundation of Korea (2023R1A2C1007293). HP and DS thank the Korea Institute for Advanced Study for its warm hospitality during their appointments as associate members. AS was partially supported by NKKP project 146401 and by the ERC Advanced Grant KnotSurf4d.

The authors used large language models in proving some statements, drafting the paper, and identifying errors. They take full responsibility for the correctness of its contents.

\section{Computation sequences and logarithmic connections}
\label{sec:computation-sequences}

We use logarithmic connections to compare cohomology along
Giraud's computation sequence, treating its endpoints and the
steps with nonzero residue.

Let \(f\colon(V,E) \to (X,0)\) be the minimal good resolution of a
rational surface singularity. For every irreducible component \(E_i\),
write \(E_i^2=-d_i\) and \(t_i=\val(E_i)\). Let \(\Sigma_i\) denote
the reduced divisor of intersection points \(E_i \cap{\overline{(E-E_i)}}\) on
\(E_i\). Recall that \(\mathcal{S}=\Theta_V(-\log{E})\),
\(E'=\sum_{t_i \ge 2}E_i\), and
\(\mathcal{L}=\mathcal{O}_V(-(K_V+E))\). Let \(r\) denote the number
of leaves of the corresponding resolution graph \(\Gamma\). We write
\(i \sim j\) when \(E_i \cap E_j \ne \emptyset\). Throughout this
section, we assume that the resolution dual graph of \(E\) has at
least two nodes. We work over a sufficiently small Stein
representative \(X\) of the singularity.

Wahl's restriction calculation~\cite[Lemma~2.5 and Propositions~2.6(2), 2.9(1)]{Wahl2017} gives
\begin{equation}\label{eq:Wahl-restriction}
 h^1(V,\mathcal{S})=r-4+h^1(V,\mathcal{S}(-E')).
\end{equation}

\subsection{Giraud divisors and computation sequences}

For a line bundle \(\mathcal{P}\) on \(V\), let \([\mathcal{P}]\) denote the divisor defined by Giraud~\cite{Giraud1982} (see also Wahl~\cite[\S4]{Wahl2015}). It is the unique integral divisor supported on \(E\) with the following properties:

\begin{enumerate}
\item \([\mathcal{P}]E_i \le c_1(\mathcal{P})E_i\) for every \(i\);

\item if an integral divisor \(D\) supported on \(E\) satisfies \(DE_i \le c_1(\mathcal{P})E_i\) for every \(i\), then \(D-[\mathcal{P}]\) is effective.
\end{enumerate}

For every integral exceptional divisor \(D\), the definition gives
\begin{equation}\label{eq:translation}
 [\mathcal{P}(-D)]=[\mathcal{P}]-D.
\end{equation}
Since \((X,0)\) is rational, Wahl~\cite[Proposition~4.4]{Wahl2015} gives
\begin{equation}\label{eq:cohomology-formula}
 h^1(V,\mathcal{P}) =\frac{[\mathcal{P}]([\mathcal{P}]+K_V)}2-[\mathcal{P}]c_1(\mathcal{P}),
\end{equation}
whenever \([\mathcal{P}]\) is effective.

For an integral exceptional divisor \(D\) on $V$, put
\begin{equation*}\label{eq:riemann-roch}
  q_{\mathcal {L}}(D) :=\frac{D(D+K_V)}2-Dc_1({\mathcal{L}}).
\end{equation*}
Since a tree with at least two nodes is neither a chain nor a star-shaped tree, Wahl's estimate \cite[Theorem~4.1 and the observation following it]{Wahl2017} gives
\(
 [\mathcal{L}] \ge E'.
\)
The reduced divisor \(E'\) is connected, its irreducible components are rational, and its dual graph is a tree. Hence
$
 E'(E'+K_V)=-2$, $-E'c_1(\mathcal{L})=E'(K_V+E)=r-2,
$
and therefore
$
 q_{\mathcal{L}}(E')=r-3.
$
Since \([\mathcal{L}]\ge E'\), \eqref{eq:cohomology-formula} applies to both \(\mathcal{L}\) and \(\mathcal{L}(-E')\). Using \eqref{eq:translation}, we obtain 
$
h^1(V,\mathcal{L})=q_{\mathcal{L}}([\mathcal{L}])$, 
$h^1(V,\mathcal{L}(-E'))=q_{\mathcal{L}}([\mathcal{L}])-q_{\mathcal{L}}(E').
$
In particular,
\begin{equation}\label{eq:twisted-L-comparison}
 h^1(V,\mathcal{L}(-E'))=h^1(V,\mathcal{L})-(r-3).
\end{equation}

For an exceptional rational cycle \(A\) and an irreducible component \(G\) of \(E\), write \(A_G\) for the coefficient of \(G\) in \(A\). We use Giraud's algorithmic description of \([\mathcal{P}]\) \cite[\S1.4]{Giraud1982} (cf. Némethi~\cite[Proposition~6.6.3]{Nemethi2022}). If \([\mathcal{P}]\) is effective and \(0 \le D<[\mathcal{P}]\) is integral, then there is a component \(G\) of \(E\) such that
$
DG>c_1(\mathcal{P})G$, $D+G \le [\mathcal{P}].
$
Iterating this step reaches \([\mathcal P]\) after finitely many steps.

Since \(E' \le [\mathcal{L}]\), applying this construction to \(\mathcal{P}=\mathcal{L}\) gives a computation sequence
\begin{equation}
\label{eq:sequence}
E'=D_0<D_1<\cdots<D_N=[\mathcal{L}],\quad
D_{j+1}=D_j+G_j\quad(0 \le j<N).
\end{equation}
Here \(G_j\) is the component added at the \(j\)-th step; the same component may occur at more than one step. The sequence satisfies
\[
 (D_j)_{G_j}<([\mathcal{L}])_{G_j},\quad
 D_jG_j>c_1(\mathcal{L})G_j=2-\val(G_j).
\]
For \(E' \le D \le [\mathcal{L}]\), put
\[
 \Phi(D) :=h^1(\mathcal{S}(-D))-h^1(\mathcal{L}(-D))-h^1(\mathcal{O}_V(-D)).
\]
For a locally free sheaf \(\mathcal{P}\) and a component \(G\) of \(E\), define
\[
 c_{\mathcal{P}}(D,G) := \dim\coker[H^0(\mathcal{P}(-D)) \to H^0(\mathcal{P}(-D)|_{G})].
\]

\begin{lemma}
\label{lem:increment}
At a step \(D \mapsto D+G\) of \eqref{eq:sequence},
\[
 \Phi(D+G)-\Phi(D)=c_{\mathcal{S}}(D,G)-c_{\mathcal{O}}(D,G).
\]
\end{lemma}

\begin{proof}
Restriction of logarithmic vector fields to \(G\) gives a natural exact sequence
\[
 0 \to \mathcal{O}_{G} \to \mathcal{S}|_{G} \to \Theta_{G}(-\log\Sigma_G) \to 0,
\]
where \(\Sigma_G\) is the reduced divisor of intersection points \(G \cap(E-G)\). By adjunction,
\[
 \Theta_{G}(-\log\Sigma_G) \simeq \mathcal{O}_{G}(-K_G-\Sigma_G) \simeq \mathcal{L}|_{G}.
\]
Tensoring with \(\mathcal{O}_V(-D)|_G\) gives
\[
 0 \to \mathcal{O}_V(-D)|_{G} \to \mathcal{S}(-D)|_{G} \to \mathcal{L}(-D)|_{G} \to 0.
\]
At a computation step, 
\[\deg(\mathcal{L}(-D)|_G)=c_1(\mathcal{L})G-DG<0,\]
so \(H^0(G,\mathcal{L}(-D)|_G)=0\) and
\begin{equation}\label{eq:cL-zero}
 c_{\mathcal{L}}(D,G)=0.
\end{equation}
The cohomology sequence of the tensored exact sequence therefore gives
\[
 H^0(G,\mathcal{O}_V(-D)|_{G}) \simeq H^0(G,\mathcal{S}(-D)|_{G}).
\]
Moreover, since \(G\) is a curve, its second coherent cohomology vanishes, so the same cohomology sequence gives
\[
h^1(G,\mathcal{S}(-D)|_G)
=h^1(G,\mathcal{O}_V(-D)|_G)+h^1(G,\mathcal{L}(-D)|_G).
\]

Since \(f\) is proper with fibers of dimension at most one, \(R^qf_*\mathcal{G}=0\) for \(q\ge2\) and every coherent \(\mathcal{G}\)~\cite{GrauertRemmert1984}. As \(X\) is Stein, Leray and Cartan's Theorem~B give
\begin{equation}\label{eq:coherent-H2-zero}
 H^2(V,\mathcal{G})=0.
\end{equation}

For each \(\mathcal{P} \in \{\mathcal{O}_V,\mathcal{S},\mathcal{L}\}\), consider the restriction exact sequence
\[
 0 \to \mathcal{P}(-D-G)
 \to \mathcal{P}(-D)
 \to \mathcal{P}(-D)|_{G}
 \to 0.
\]
The associated cohomology sequence and \eqref{eq:coherent-H2-zero} give
\[
 h^1(V,\mathcal{P}(-D-G)) -h^1(V,\mathcal{P}(-D))
 = c_{\mathcal{P}}(D,G) -h^1(G,\mathcal{P}(-D)|_{G}).
\]
Subtracting the formulas for \(\mathcal{L}\) and \(\mathcal{O}_V\) from the formula for \(\mathcal{S}\), and using the cohomology identity above and \eqref{eq:cL-zero}, gives the assertion.
\end{proof}

\begin{lemma}
For the first term \(D_0=E'\) of the computation sequence \eqref{eq:sequence}, we have
\begin{equation}\label{eq:initial-O}
 H^1(V,\mathcal{O}_V(-E'))=0.
\end{equation}
Consequently,
\begin{equation}\label{eq:initial-Phi}
 \Phi(E')=h^1(V,\mathcal{S}(-E'))-h^1(V,\mathcal{L}(-E')).
\end{equation}
\end{lemma}

\begin{proof}
Since \(E'\) is connected and reduced, \(H^0(E',\mathcal{O}_{E'})=\mathbb{C}\), and the restriction map from \(H^0(V,\mathcal{O}_V)\) is surjective. The exact sequence
\[
0\to\mathcal{O}_V(-E')\to\mathcal{O}_V\to\mathcal{O}_{E'}\to0
\]
and rationality give \(H^1(V,\mathcal{O}_V(-E'))=0\). Equation~\eqref{eq:initial-Phi} follows from the definition of \(\Phi\).
\end{proof}

\subsection{Logarithmic connections and terminal cohomology}

\begin{lemma}
\label{lem:terminal-positivity}
Every coefficient of the exceptional rational cycle \([\mathcal{L}]-c_1(\mathcal{L})\) is strictly positive.
\end{lemma}

\begin{proof}
We use the intersection pairing to regard \(c_1(\mathcal{L})\) as an exceptional rational cycle. By the definition of \([\mathcal{L}]\), we have
\[
 ([\mathcal{L}]-c_1(\mathcal{L}))E_i \le 0
\]
for every \(i\).

This cycle is nonzero. Indeed, if \([\mathcal{L}]=c_1(\mathcal{L})\), then \(\mathcal{L}=\mathcal{O}_V(-K_V-E)\) gives
\[
 Z_K=E+[\mathcal{L}],
\]
where \(Z_K\) is the numerical anticanonical cycle. This cycle would be integral, so the singularity would be numerically Gorenstein and rational. Its minimal good resolution graph would then be an \(ADE\) graph by N\'emethi~\cite[Example~7.1.5(d)]{Nemethi2022}. Such a graph has at most one node, contrary to our assumption.

Negative definiteness and connectedness now imply that all coefficients are positive~\cite[Corollary~2.1.19(a)]{Nemethi2022}.
\end{proof}

\begin{lemma}
\label{lem:logarithmic-connection}
Let \(\mathcal{P}\) be a line bundle on \(V\), and write its first Chern class as an exceptional rational cycle
\[
 c_1(\mathcal{P})=\sum_i b_iE_i.
\]
Then \(\mathcal{P}\) admits an integrable logarithmic connection with residue \(-b_i\) along \(E_i\).
\end{lemma}

\begin{proof}
Choose a positive integer \(n\) such that \(nc_1(\mathcal{P})\) is an integral exceptional divisor. For a rational surface singularity, the Chern class map
\[
 c_1 \colon \operatorname{Pic}(V) \to H^2(V,\mathbb{Z})
\]
is an isomorphism by N\'emethi~\cite[Proposition~7.1.10]{Nemethi2022}. The equality of Chern classes therefore gives
\[
 \mathcal{P}^{\otimes n} \simeq \mathcal{O}_V(nc_1(\mathcal{P})).
\]
The trivial connection on \(\mathcal{O}_V\) induces an integrable logarithmic connection on \(\mathcal{O}_V(nc_1(\mathcal{P}))\), with residue \(-nb_i\) along \(E_i\); see Esnault--Viehweg~\cite[Lemma~2.7]{EsnaultViehweg1992}.

Let \(\nabla^{(n)}\) denote the induced connection on
\(\mathcal{P}^{\otimes n}\).  Choose an open cover \(\{U_\alpha\}\) of
\(V\) such that \(\mathcal{P}\) has a nowhere vanishing holomorphic
section \(s_\alpha\) on each \(U_\alpha\). We can therefore write
\[
 \nabla^{(n)}(s_\alpha^{\otimes n})
 =s_\alpha^{\otimes n} \otimes A_\alpha,
\]
where \(A_\alpha\) is a logarithmic one-form on \(U_\alpha\). On
overlaps, \(s_\beta=g_{\alpha\beta}s_\alpha\), so
\[
A_\beta=A_\alpha+n\,dg_{\alpha\beta}/g_{\alpha\beta}.
\]
Thus the forms \(A_\alpha/n\) define a logarithmic connection on \(\mathcal{P}\). It is integrable because \(dA_\alpha=0\), and its residue along \(E_i\) is \(-b_i\).
\end{proof}
Fix a sufficiently small Stein representative \(X\), with resolution \(f \colon V \to X\), and let \(j \colon V-E \hookrightarrow V\).

\begin{lemma}
\label{lem:local-system-vanishing}
For every rank-one complex local system \(\mathbb{V}\) on \(V-E\), we have
\[
 R\Gamma(V,j_!\mathbb{V})=0.
\]
\end{lemma}

\begin{proof}
We describe the neighborhoods used to establish this vanishing
result. Choose \(X\) as the intersection of a local analytic model of
the singularity with a sufficiently small open ball of radius
\(\varepsilon\). For \(0<\delta<\varepsilon\), let \(X_\delta \subset
X\) be its intersection with the concentric open ball of radius
\(\delta\), and put
\[
 W_\delta=f^{-1}(X_\delta).
\]
The local conical structure of an isolated complex singularity identifies these punctured neighborhoods with products of the same link and radial intervals. In particular, for \(\varepsilon\) sufficiently small, the inclusions
\[
 X_\delta-\{0\} \hookrightarrow X-\{0\}
\]
are homotopy equivalences. Since \(f\) is an isomorphism away from the exceptional divisor, the corresponding inclusions
\[
 W_\delta-E \hookrightarrow V-E
\]
are also homotopy equivalences, and induce isomorphisms on cohomology with coefficients in the local system \(\mathbb{V}\).

The neighborhoods \(W_\delta\) are cofinal among neighborhoods of \(E\): if \(W\supset E\) is open, properness makes \(f(V\setminus W)\) closed and disjoint from \(0\), so \(X_\delta\cap f(V\setminus W)=\emptyset\) for all sufficiently small \(\delta\), and hence \(W_\delta\subset W\).

Now put \(K=Rj_*\mathbb{V}\), and let \(i \colon E \hookrightarrow V\). For every sufficiently small \(\delta\), the natural identifications
\[
 R\Gamma(V,K) \simeq R\Gamma(V-E,\mathbb{V}),
 \quad
 R\Gamma(W_\delta,K) \simeq R\Gamma(W_\delta-E,\mathbb{V})
\]
show that the restriction morphism
\[
 R\Gamma(V,K) \to R\Gamma(W_\delta,K)
\]
is a quasi-isomorphism. Since \(E\) is compact, cohomology on \(E\) with coefficients in \(i^{-1}K\) is obtained as the direct limit of cohomology over its neighborhoods; see Kashiwara--Schapira~\cite[Proposition~2.5.1 and Remark~2.6.9]{KashiwaraSchapira1990}. Using the cofinal system \(W_\delta\), we conclude that the natural restriction morphism
\[
 R\Gamma(V,K) \to R\Gamma(E,i^{-1}K)
\]
is a quasi-isomorphism.

Since \(j^{-1}K \simeq \mathbb{V}\), the localization triangle is
%
%
\[
\begin{tikzcd}[row sep=2.5em, column sep=2em]
	j_!\mathbb{V} \arrow[rr] & & K \arrow[dl] \\
	& i_*i^{-1}K \arrow[ul, "+1" description] &
\end{tikzcd}
\]
Applying \(R\Gamma(V,-)\), the map between the second and third terms becomes the quasi-isomorphism just established. The first term must therefore vanish, which proves the lemma.
\end{proof}

\begin{lemma}
\label{lem:positive-residue-comparison}
Let \(\mathcal{P}\) be a line bundle on \(V\) with an integrable logarithmic connection whose residues along all components of \(E\) are positive. If \(H^1(V,\mathcal{P})=0\), then the logarithmic de Rham differential induces an isomorphism
\[
 H^1(V,\Omega_V^1(\log E) \otimes \mathcal{P})
 \to H^1(V,\Omega_V^2(\log E) \otimes \mathcal{P}).
\]
\end{lemma}

\begin{proof}
Write \(\nabla\) for the connection on \(\mathcal{P}\). Its logarithmic de Rham complex is
\[
 \operatorname{DR}(\mathcal{P},\nabla)
 = [\mathcal{P} \to \Omega_V^1(\log E) \otimes \mathcal{P}
 \to \Omega_V^2(\log E) \otimes \mathcal{P}],
\]
placed in degrees \(0,1,2\). Its differentials are induced by \(\nabla\) and are \(\mathbb{C}\)-linear.

Let \(j \colon V-E \hookrightarrow V\) be the inclusion, and let
\(\mathbb{V}\) be the local system of horizontal sections of
\(\nabla\) on \(V-E\).
Deligne's logarithmic de Rham comparison theorem, together with
Poincar\'e--Verdier duality, identifies the logarithmic de Rham
complex with the extension by zero of its horizontal local system
when no residue is a nonpositive integer; see
Deligne~\cite[II, Proposition~3.13]{Deligne1970} and the
extension-by-zero formulation in
Esnault--Viehweg~\cite[\S2.11]{EsnaultViehweg1992}.
Since all residues of \(\nabla\) are positive, the natural morphism
\[
j_!\mathbb{V} \longrightarrow
\operatorname{DR}(\mathcal{P},\nabla)
\]
is a quasi-isomorphism.
Lemma~\ref{lem:local-system-vanishing} therefore gives
\[
 \mathbb{H}^k(V,\operatorname{DR}(\mathcal{P},\nabla))=0
\]
for every \(k\).

The vanishing used in Equation~\eqref{eq:coherent-H2-zero} gives
\[
 H^q(V,\mathcal{G})=0 \quad(q \ge 2)
\]
for every coherent sheaf \(\mathcal{G}\) on \(V\). Consider the hypercohomology spectral sequence
\[
 E_1^{p,q}=H^q(V,\Omega_V^p(\log E) \otimes \mathcal{P})
 \Longrightarrow
 \mathbb{H}^{p+q}(V,\operatorname{DR}(\mathcal{P},\nabla)).
\]
Only the columns \(p=0,1,2\) and the rows \(q=0,1\) can be
nonzero. Since \(H^1(V,\mathcal{P})=0\), the terms
\(E_2^{1,1}\) and \(E_2^{2,1}\) are the kernel and cokernel
of the map in the statement. No higher differential enters or
leaves either term: incoming differentials have source in a
row \(q\ge2\), and outgoing differentials have target in a
column \(p\ge3\). Thus \(E_2^{1,1}=E_\infty^{1,1}=0\) and
\(E_2^{2,1}=E_\infty^{2,1}=0\), by the vanishing of
hypercohomology.
\end{proof}

\begin{lemma}
\label{lem:terminal}
For the last term \(D_N=[\mathcal{L}]\) of \eqref{eq:sequence}, \(\Phi([\mathcal{L}])=0\).
\end{lemma}

\begin{proof}
Equation~\eqref{eq:translation} gives
\[
 [\mathcal{L}(-[\mathcal{L}])]=[\mathcal{L}]-[\mathcal{L}]=0.
\]
Applying Equation~\eqref{eq:cohomology-formula} to \(\mathcal{L}(-[\mathcal{L}])\), we therefore obtain
\[
 H^1(V,\mathcal{L}(-[\mathcal{L}]))=0.
\]
By Lemmas~\ref{lem:terminal-positivity} and \ref{lem:logarithmic-connection}, the line bundle \(\mathcal{L}(-[\mathcal{L}])\) admits an integrable logarithmic connection whose residues are the strictly positive coefficients of \([\mathcal{L}]-c_1(\mathcal{L})\). Lemma~\ref{lem:positive-residue-comparison} applies. Using the identities
\[
 \Omega_V^1(\log E) \otimes \mathcal{L} \simeq \mathcal{S},
 \quad
 \Omega_V^2(\log E) \otimes \mathcal{L} \simeq \mathcal{O}_V,
\]
we obtain
\[
 H^1(V,\mathcal{S}(-[\mathcal{L}]))
 \simeq H^1(V,\mathcal{O}_V(-[\mathcal{L}])).
\]
Together with \(H^1(V,\mathcal{L}(-[\mathcal{L}]))=0\), this proves \(\Phi([\mathcal{L}])=0\).
\end{proof}

\subsection{Nonzero residues}

Write
\[
 c_1(\mathcal{L})=\sum_i a_iE_i.
\]
For each \(0 \le j<N\), Lemma~\ref{lem:logarithmic-connection} gives an integrable logarithmic connection \(\nabla_j\) on \(\mathcal{L}(-D_j)\). Its residue along \(E_i\) is
\[
 \rho_{E_i}(D_j):=(D_j)_{E_i}-a_i.
\]

At a step with nonzero residue, the next lemma and
Lemma~\ref{lem:increment} show that
\(\Phi(D_{j+1})\ge\Phi(D_j)\). We treat the zero-residue block
in Section~\ref{sec:cohomological-comparison}.

\begin{lemma}
\label{lem:nonzero}
Let \(0 \le j<N\). If \(\rho_{G_j}(D_j) \ne 0\), then
\[
 c_{\mathcal{S}}(D_j,G_j) \ge c_{\mathcal{O}}(D_j,G_j).
\]
\end{lemma}

\begin{proof}
As in the proof of Lemma~\ref{lem:increment}, \(\deg(\mathcal{L}(-D_j)|_{G_j})<0\), and the restriction sequence identifies
\[
H^0(G_j,\mathcal{S}(-D_j)|_{G_j})
\simeq H^0(G_j,\mathcal{O}_V(-D_j)|_{G_j}).
\]

The logarithmic de Rham complex of \((\mathcal{L}(-D_j),\nabla_j)\) has the form
\[
 \mathcal{L}(-D_j)
 \xrightarrow{\nabla_j}
 \mathcal{S}(-D_j)
 \xrightarrow{\nabla_j^1}
 \mathcal{O}_V(-D_j).
\]
The second differential preserves sections vanishing along \(G_j\), and hence induces a map on the restrictions to \(G_j\). Let \(r_{\mathcal{S},j}\) and \(r_{\mathcal{O},j}\) denote the restriction maps on global sections. We obtain a commutative diagram
\[
\begin{tikzcd}[column sep=large]
 H^0(V,\mathcal{S}(-D_j))
 \arrow[r,"r_{\mathcal{S},j}"]
 \arrow[d,"\nabla_j^1"']
 &
 H^0(G_j,\mathcal{S}(-D_j)|_{G_j})
 \arrow[d,"\overline{\nabla}_j^1"]
 \\
 H^0(V,\mathcal{O}_V(-D_j))
 \arrow[r,"r_{\mathcal{O},j}"']
 &
 H^0(G_j,\mathcal{O}_V(-D_j)|_{G_j}).
\end{tikzcd}
\]

On the subbundle
\[
 \mathcal{O}_V(-D_j)|_{G_j} \subset \mathcal{S}(-D_j)|_{G_j},
\]
the induced differential is, up to sign, multiplication by the residue \(\rho_{G_j}(D_j)\); see Esnault--Viehweg~\cite[Lemma~2.6(b)]{EsnaultViehweg1992}. Under the preceding identification of section spaces, \(\overline{\nabla}_j^1\) is multiplication by the nonzero scalar \(\pm\rho_{G_j}(D_j)\), hence is an isomorphism.

Since the right-hand vertical map is an isomorphism, the
diagram induces a surjection
\(\coker r_{\mathcal{S},j}\twoheadrightarrow
\coker r_{\mathcal{O},j}\), which gives the inequality.
\end{proof}

\section{A comparison of logarithmic cohomology groups}
\label{sec:cohomological-comparison}

We retain the notation and hypotheses of Section~\ref{sec:computation-sequences} until the final subsection, where we allow arbitrary resolution graphs. The proof of the following proposition is at the end of Subsection~\ref{subsec:zero-residue-cohomology}.

\begin{proposition}\label{prop:cohomological-comparison}
If the resolution graph of a rational surface singularity is neither a chain nor a star-shaped graph, then
\[
 h^1(V,\mathcal{S}(-E')) \le h^1(V,\mathcal{L}(-E')).
\]
\end{proposition}

\subsection{Construction of the zero-residue divisor}

For a rational cycle \(x\), write \(x=x_+-x_-\) for its positive and negative parts. Write
\[
c_1(\mathcal{L}(-E'))=\sum_i b_iE_i,\qquad
\mathcal{Z}=\{i:b_i\in\mathbb{Z}_{\ge0}\},\qquad
C=\lceil c_1(\mathcal{L}(-E'))\rceil_+,
\]
where the ceiling is taken coefficientwise. If \(D=E'+J\), the residue identity is
\[
 \rho_{E_i}(E'+J)=J_i-b_i.
\]
By \eqref{eq:translation} and Lemma~\ref{lem:terminal-positivity},
\begin{equation}\label{eq:W-strict}
 ([\mathcal{L}]-E')_i-b_i>0\quad\text{for every }i.
\end{equation}
Since \([\mathcal{L}]-E'\) is integral and effective, \(C\le[\mathcal{L}]-E'\).

We now construct a computation sequence from \(0\) to \([\mathcal{L}]-E'\) in the following way. Recall that the addition of \(E_i\) is permitted precisely when
\[
 JE_i>c_1(\mathcal{L}(-E'))E_i
 \Longleftrightarrow
 (c_1(\mathcal{L}(-E'))-J)E_i<0.
\]
Suppose first that \(0 \le J<C\), and put \(y=c_1(\mathcal{L}(-E'))-J\). Then \(y_+ \ne 0\). If no index in \(\Supp(y_+)\) were permitted, one would have \(yE_i \ge 0\) on that support. Writing \(y=y_+-y_-\) with disjoint positive and negative parts, we obtain
\[
 y_+E_i=yE_i+y_-E_i \ge 0
\]
for \(i \in \Supp(y_+)\), and hence
\[
 y_+^2=\sum_i(y_+)_i(y_+E_i) \ge 0,
\]
contrary to negative definiteness. We may therefore add a component with \(y_i>0\). Since \(J_i<b_i\) and \(J_i\) is integral, \(J+E_i\le C\); this step has residue \(J_i-b_i=-y_i<0\). Iteration reaches \(C\), with every intermediate cycle bounded by \(C \le [\mathcal{L}]-E'\).

At \(J=C\), the zero-residue components are exactly the \(E_i\) with \(i \in \mathcal{Z}\). If \(C=[\mathcal{L}]-E'\), then \eqref{eq:W-strict} implies \(\mathcal{Z}=\emptyset\), since \(C_i=b_i\) for \(i \in \mathcal{Z}\). In this case the sequence is complete, with no zero- or positive-residue steps.

It remains to describe the construction when \(C<[\mathcal{L}]-E'\). For every \(i \in \mathcal{Z}\), \eqref{eq:W-strict} and \(b_i \in \mathbb{Z}\) imply
\begin{equation}\label{eq:zero-coefficient-bound}
 ([\mathcal{L}]-E')_i \ge C_i+1,
\end{equation}
so every zero-residue component \(E_i\) may be added once without
passing \([\mathcal{L}]-E'\). Moreover, every connected component of
the induced subgraph on \(\mathcal{Z}\) has an edge to its
complement. Otherwise a component without a boundary edge would, by
connectedness of the resolution graph, contain every vertex. This
would imply that \(c_1(\mathcal{L}(-E'))\) is integral and effective,
which in turn implies that the graph is numerically Gorenstein, thus
ADE, contradicting the assumption that the graph has at least two
nodes, as in the proof of Lemma~\ref{lem:terminal-positivity}.

Put \(x=C-c_1(\mathcal{L}(-E'))\). Its coefficients vanish on \(\mathcal{Z}\) and are positive off \(\mathcal{Z}\). At a boundary vertex \(i\) of a connected component of the subgraph on \(\mathcal{Z}\) one has
\[
 xE_i=\sum_{j \notin \mathcal{Z}}x_jE_jE_i>0,
\]
so \(E_i\) may be added. Root the connected component at \(i\). After \(E_i\) has been added, an adjacent vertex \(j\) that has not yet been treated has coefficient zero in the current difference from \(c_1(\mathcal{L}(-E'))\), while the coefficient at an already treated neighbor is one; all other coefficients are nonnegative. The intersection of \(E_j\) with that difference is therefore positive, so \(E_j\) may be added. Proceeding away from the root adds \(E_j\) once for every vertex \(j\) of the connected component. By \eqref{eq:zero-coefficient-bound}, the resulting cycle
\[
 J^+=C+\sum_{i \in \mathcal{Z}}E_i
\]
still satisfies \(J^+ \le [\mathcal{L}]-E'\).

All residues at \(J^+\) are positive. If \(J^+<[\mathcal{L}]-E'\), Giraud's construction applied to \(\mathcal{P}=\mathcal{L}(-E')\) completes the sequence up to \([\mathcal{L}]-E'\). The residues \(J_i-b_i\) remain positive as coefficients increase. The sequence therefore has the following order:

\begin{enumerate}
\item additions from \(0\) to \(C\), each with negative residue; \item one addition of \(E_i\) for every \(i \in \mathcal{Z}\), each with zero residue; \item the remaining additions up to \([\mathcal{L}]-E'\), each with positive residue.
\end{enumerate}

After translating every cycle \(J\) by \(E'\), we henceforth take the computation sequence \eqref{eq:sequence} to be the ordered sequence just constructed from \(E'\) to \([\mathcal{L}]\).

Put
\[
 D^- := E'+C,\quad
 \mathcal{R}=\sum_{i \in \mathcal{Z}}E_i,\quad
 D^+ := D^-+\mathcal{R}.
\]

The cycles \(D^-\) and \(D^+\) occur immediately before and after the zero-residue steps, respectively. The following vanishing will be used in the next subsection.

\begin{lemma}
\label{lem:zero-vanishing}
For the divisor \(D^+\) defined above, we have
\[
 H^1(V,\mathcal{O}_V(-D^+))=0.
\]
\end{lemma}

\begin{proof}
Let \(Z_K=\sum z_iE_i\) denote the numerical anticanonical cycle as before, with \(z_i>0\); see Némethi~\cite[Corollary~2.1.19]{Nemethi2022}. Put
\[
 Z_0 :=\lfloor Z_K\rfloor,\quad
 P_0 :=\sum_{\substack{\text{\(E_i\) is not a leaf}\\0<z_i<1}}E_i.
\]
The relation \(c_1(\mathcal{L}(-E'))=Z_K-E-E'\) gives
\begin{equation}\label{eq:Dplus}
 D^+=Z_0+P_0.
\end{equation}
This follows coefficientwise from \(b_i=z_i-1\) on leaves
and \(b_i=z_i-2\) on the other components.

Write \(Z_K=Z_0+F_0\), with all coefficients of \(F_0\) in \([0,1)\). Apply Fujino's analytic relative vanishing theorem for surfaces (\cite[Theorem~11.3(2)]{Fujino2021}) to the smooth klt pair \((V,F_0)\), the line bundle \(\mathcal{O}_V(-Z_0)\), the auxiliary divisor \(D=0\), and the resolution map \(f:V \to X\). The numerical condition holds because
\[
 (-Z_0)E_i+(D-(K_V+F_0))E_i =-(K_V+Z_K)E_i=0
\]
for every exceptional component. Thus
\[
 R^1f_*\mathcal{O}_V(-Z_0)=0,\quad
 H^1(V,\mathcal{O}_V(-Z_0))=0.
\]

It remains to add \(P_0\). Set
\[
 \Sigma=\bigcup_{z_i \ge 1}E_i=\Supp Z_0.
\]
The anticanonical equations are
\begin{equation}\label{eq:concavity}
 \sum_{j \sim i}z_j=2-d_i(1-z_i).
\end{equation}
For \(z_i<1\), they imply
\[
 \sum_{j \sim i}z_j \le 2z_i<2.
\]
Along a path, this inequality expresses discrete concavity. Hence \(\Sigma\) is connected (possibly empty): a path between two distinct components would have endpoint values at least one and internal values below one, contrary to concavity.

Assume first that \(\Sigma \ne \emptyset\). Every connected component of the subgraph determined by \(\Supp P_0\) has a unique edge to \(\Sigma\). Uniqueness follows because the graph is a tree. To see existence when \(\Sigma\) contains a component that is not a leaf, join a vertex in \(\Supp P_0\) to a vertex of \(\Sigma\) that is not a leaf. The corresponding path in the dual graph lies in the subgraph determined by \(E'\). The first vertex of this path outside the starting component has a positive coefficient, is not a leaf, and does not lie in \(\Supp P_0\); hence it lies in \(\Sigma\). If \(\Sigma\) contains only leaves, its connectedness forces it to consist of a single leaf. Then every vertex that is not a leaf has coefficient in \((0,1)\), so \(P_0=E'\), which meets that leaf.

Let \(q\) be a vertex in \(\Supp P_0\) adjacent to \(\Sigma\). Since \(z_q<1\), the inequality following from \eqref{eq:concavity} shows that the coefficient of the adjacent vertex in \(\Sigma\) is less than two. It is at least one, so its coefficient in \(Z_0\) is exactly one. Every other boundary neighbor of that connected component of \(\Supp P_0\) is a leaf with coefficient below one and hence has coefficient zero in \(Z_0\).

In each connected component of \(\Supp P_0\), first add the vertex adjacent to \(\Sigma\), and then add the remaining vertices in order away from \(\Sigma\). For the first vertex, the only nonzero adjacent contribution is the coefficient of that adjacent component; for every later vertex it is the coefficient of the previously added predecessor. Immediately before \(E_i\) is added, the current cycle \(Z\) therefore satisfies \(ZE_i=1\). Hence the quotient in
\[
 0 \to \mathcal{O}_V(-Z-E_i) \to \mathcal{O}_V(-Z) \to \mathcal{O}_{E_i}(-Z) \to 0
\]
is \(\mathcal{O}_{\mathbb{P}^1}(-1)\); its \(H^0\) and \(H^1\) vanish, so \(H^1\) is unchanged. If \(\Sigma\) is empty, then \(Z_0=0\) and \(P_0=E'\), and the result follows from \eqref{eq:initial-O}. Equation~\eqref{eq:Dplus} proves the lemma.
\end{proof}

\subsection{Cohomology on the zero-residue divisor}
\label{subsec:zero-residue-cohomology}

The goal of this subsection is the following proposition, whose proof is given after Lemma~\ref{lem:zero-injection}.

\begin{proposition}
\label{prop:zero-comparison}
For the cycles \(D^-\) and \(D^+\) defined above, we have
\begin{equation*}\label{eq:zero-increment}
 \Phi(D^+)-\Phi(D^-) \ge 0.
\end{equation*}
\end{proposition}

We first compare the spaces of sections on \(\mathcal{R}\). For \(\mathcal{P} \in \{\mathcal{O}_V,\mathcal{L},\mathcal{S}\}\), set
\[
 Q_{\mathcal{P}} :=\mathcal{P}(-D^-)|_{\mathcal{R}},\quad
 I_{\mathcal{P}} :=
 \im[H^0(\mathcal{P}(-D^-)) \to H^0(Q_{\mathcal{P}})].
\]

For each connected component \(R_\alpha\) of \(\mathcal{R}\), let \(I_\alpha=\{i:E_i \subset R_\alpha\}\). For \(i \in I_\alpha\), put
\[
 \beta_i=D^-E_i-c_1(\mathcal{L})E_i,\quad
 M_i=\mathcal{O}_V(-D^-)|_{E_i}.
\]
With \(c_1(\mathcal{L})=\sum_j a_jE_j\) as above, all residues on \(R_\alpha\) vanish, and hence
\begin{equation}\label{eq:beta-boundary}
 \beta_i
 =\sum_{\substack{j \notin I_\alpha\\j \sim i}}
 ((D^-)_j-a_j)
 =\sum_{\substack{j \notin I_\alpha\\j \sim i}}
 \rho_{E_j}(D^-).
\end{equation}
Every residue in this boundary sum is positive. Thus \(\beta_i \ge 0\), with strict inequality exactly when \(E_i\) meets \(E-R_\alpha\). Restriction of logarithmic vector fields to \(E_i\), followed by tensoring with \(M_i\), gives the exact sequence
\begin{equation}\label{eq:component-split}
 0 \to M_i \to Q_{\mathcal{S}}|_{E_i} \to \mathcal{L}(-D^-)|_{E_i} \to 0.
\end{equation}
The two line bundles in this sequence have degrees
\[
 \deg M_i=t_i-2-\beta_i,\quad
 \deg\mathcal{L}(-D^-)|_{E_i}=-\beta_i.
\]

The following lemma describes sections on a single component. Recall that \(\Sigma_i\) denotes the reduced divisor of intersection points \(E_i \cap \overline{E-E_i}\).

\begin{lemma}
\label{lem:zero-component-sections}
Let \(i \in I_\alpha\). If \(\beta_i>0\), every global section of \(Q_{\mathcal{S}}|_{E_i}\) has zero image in \(H^0(E_i,\mathcal{L}(-D^-)|_{E_i})\). If \(\beta_i=0\), choose an affine coordinate \(z_i\) on \(E_i\) such that the chart contains \(\Sigma_i\), and put
\[
 a_{ij}=z_i(E_i \cap E_j),\quad
 P_i(z_i)=\prod_{j \sim i}(z_i-a_{ij}).
\]
One can choose a normal coordinate \(u_i\) to first order along this chart and a local generator \(\eta_i\) of \(M_i\) such that every global section of \(Q_{\mathcal{S}}|_{E_i}\) is represented by
\[
 \eta_i\left(B_i(z_i)u_i\partial_{u_i}+\tau_iP_i(z_i)\partial_{z_i}\right),
\]
where \(\tau_i \in \mathbb{C}\) and \(B_i\) is a polynomial satisfying
\[
 \deg B_i \le t_i-1,\quad
 [z_i^{t_i-1}]B_i=-d_i\tau_i.
\]
The coordinates can be chosen so that every adjacent branch is a fiber over its intersection point to first order.
\end{lemma}

\begin{proof}
If \(\beta_i>0\), the target line bundle in \eqref{eq:component-split} has negative degree. It therefore has no nonzero global section, which proves the first assertion.

Suppose that \(\beta_i=0\). We first choose a retraction of the first-order neighborhood of \(E_i\) onto \(E_i\) that is compatible with the adjacent branches. Such retractions exist locally in normal-crossing coordinates. Two local retractions differ by a derivation with values in the conormal bundle. Compatibility with the adjacent branches requires this derivation to vanish at \(\Sigma_i\). Consequently, the obstruction to gluing the local retractions lies in
\[
 H^1\left(E_i,\Theta_{E_i}(-\log{\Sigma_i})\otimes N_{E_i/V}^{\vee}\right);
\]
compare \cite[\S3.4]{Sernesi2006}. Since
\[
 \Theta_{E_i}(-\log{\Sigma_i})\otimes N_{E_i/V}^{\vee}
 \simeq \mathcal{O}_{\mathbb{P}^1}(d_i+2-t_i)
\]
and rationality gives \(d_i+2-t_i \ge 1\), this cohomology group vanishes. We may therefore choose such a retraction. In local coordinates with \(E_i=(u=0)\) and an adjacent branch given by \((z=0)\), its pullback of the base coordinate has the form
\[
 z+uh(z)\pmod{u^2},\quad h(0)=0.
\]
Thus the adjacent branch is a fiber over its marked point to first order.

Put \(w_i=z_i^{-1}\). Since \(N_{E_i/V}^{\vee} \simeq \mathcal{O}_{\mathbb{P}^1}(d_i)\) and \(M_i \simeq \mathcal{O}_{\mathbb{P}^1}(t_i-2)\), choose a section \(u_i\) of \(N_{E_i/V}^{\vee}\) with divisor \(d_i\infty\) and a meromorphic section \(\eta_i\) of \(M_i\) with divisor \((t_i-2)\infty\). These sections have neither zeros nor poles on the affine chart. Using the retraction, regard \(u_i\) as a normal coordinate to first order along that chart. Near infinity, the corresponding local generators \(v_i\) and \(\widetilde\eta_i\) satisfy
\[
 v_i=z_i^{d_i}u_i\pmod{u_i^2},\quad
 \eta_i=w_i^{t_i-2}\widetilde\eta_i.
\]
After restriction to \(E_i\), these transition formulas give
\[
 u_i\partial_{u_i}=v_i\partial_{v_i},\quad
 \partial_{z_i}=-w_i^2\partial_{w_i}+d_iw_iv_i\partial_{v_i}.
\]

The image of \(\eta_iP_i(z_i)\partial_{z_i}\) generates the tangential quotient globally: near infinity it is represented by
\[
-\widetilde{\eta}_i w_i^{t_i}P_i(w_i^{-1})\partial_{w_i},
\]
and \(w_i^{t_i}P_i(w_i^{-1})\) has value \(1\) at \(w_i=0\). Hence every section has the stated affine form, with \(\tau_i\in\mathbb{C}\).

The section is consequently represented near infinity by
\[
 \widetilde\eta_i\left(\{w_i^{t_i-2}B_i(w_i^{-1})+d_i\tau_iw_i^{t_i-1}P_i(w_i^{-1})\}v_i\partial_{v_i}-\tau_iw_i^{t_i}P_i(w_i^{-1})\partial_{w_i}\right).
\]
The tangential coefficient is regular. Since \(P_i\) is monic of degree \(t_i\), regularity of the normal coefficient implies that \(B_i\) is a polynomial of degree at most \(t_i-1\). The coefficient of \(w_i^{-1}\) must also vanish, which gives \([z_i^{t_i-1}]B_i=-d_i\tau_i\).
\end{proof}

We next use the matching conditions at the intersection points to show that the tangential parts vanish simultaneously on all components.

\begin{lemma}
\label{lem:zero-tangential-vanishing}
For every connected component \(R_\alpha\) of \(\mathcal{R}\), the natural map
\[
 H^0(R_\alpha,Q_{\mathcal{S}}|_{R_\alpha})
 \to \bigoplus_{i \in I_\alpha}H^0(E_i,\mathcal{L}(-D^-)|_{E_i})
\]
is zero.
\end{lemma}

\begin{proof}
Fix a global section on \(R_\alpha\). Its tangential part already vanishes on every component with \(\beta_i>0\), by Lemma~\ref{lem:zero-component-sections}. For a component with \(\beta_i=0\), use the coordinates, local generator, and coefficients \(B_i,\tau_i\) from that lemma. Notice that \(\beta_i=0\) implies that every neighbor of \(E_i\) belongs to \(R_\alpha\), by \eqref{eq:beta-boundary}.

Consider an intersection point \(p_{ij}=E_i \cap E_j\) with \(\beta_i=\beta_j=0\). Choose normal-crossing coordinates \((x,y)\) with \(E_i=(x=0)\) and \(E_j=(y=0)\). In the logarithmic tangent fiber at this point, the normal and tangential Euler directions on \(E_i\) are \(x\partial_x\) and \(y\partial_y\), respectively; on \(E_j\) their roles are reversed. Multiplying a branch parameter by a unit does not change these directions in the fiber. Write
\[
 \eta_j(p_{ij})=\gamma_{ij}\eta_i(p_{ij}),\quad
 \gamma_{ij} \in \mathbb{C}^*,\quad
 \gamma_{ji}=\gamma_{ij}^{-1}.
\]
The two restrictions of our section then have values
\begin{align*}
 &\eta_i(p_{ij})(B_i(a_{ij})x\partial_x+\tau_iP_i'(a_{ij})y\partial_y),\\
 &\gamma_{ij}\eta_i(p_{ij})(\tau_jP_j'(a_{ji})x\partial_x+B_j(a_{ji})y\partial_y).
\end{align*}
Equality of these values gives
\begin{equation}\label{eq:matching}
 \begin{aligned}
 B_i(a_{ij})&=\gamma_{ij}P_j'(a_{ji})\tau_j,\\
 P_i'(a_{ij})\tau_i&=\gamma_{ij}B_j(a_{ji}).
 \end{aligned}
\end{equation}
If instead \(\beta_i=0\) and \(\beta_j>0\), the tangential part on \(E_j\) is zero. Comparing its value with the normal direction on \(E_i\) therefore gives \(B_i(a_{ij})=0\).

For an edge with \(\beta_i=\beta_j=0\), put
\[
 \lambda_{ij}=-\gamma_{ij}\frac{P_j'(a_{ji})}{P_i'(a_{ij})}.
\]
The intersection points on each component are distinct, so these scalars are nonzero, and \(\gamma_{ji}=\gamma_{ij}^{-1}\) gives \(\lambda_{ij}\lambda_{ji}=1\). Since \(\deg B_i \le t_i-1\), Lagrange interpolation gives
\[
 [z_i^{t_i-1}]B_i=\sum_{j \sim i}\frac{B_i(a_{ij})}{P_i'(a_{ij})}.
\]
Combining this identity with Lemma~\ref{lem:zero-component-sections}, \eqref{eq:matching}, and the vanishing of \(B_i(a_{ij})\) when \(\beta_j>0\), we obtain
\begin{equation}\label{eq:tau-system}
 -d_i\tau_i+
 \sum_{\substack{j \sim i\\\beta_j=0}}\lambda_{ij}\tau_j=0,\quad
 \lambda_{ij}\lambda_{ji}=1.
\end{equation}

The subgraph on the vertices with \(\beta_i=0\) is a forest. In each connected component, choose a root and nonzero scalars \(c_i\) recursively by
\[
 c_j=c_i/\lambda_{ij}
\]
whenever \(i\) is the parent of \(j\). Substitute \(\tau_i=c_i\sigma_i\) into \eqref{eq:tau-system} and divide the \(i\)-th equation by \(c_i\). The resulting system has coefficient matrix \((E_iE_j)_{\beta_i=\beta_j=0}\), a negative definite principal submatrix of the intersection matrix. Hence every \(\sigma_i\), and therefore every \(\tau_i\), is zero.
\end{proof}

The next lemma gives the canonical map that will be applied to these sections. For a reduced divisor \(R\) supported on \(E\), write \(\Theta_R=\operatorname{Der}_{\mathbb{C}}(\mathcal{O}_R,\mathcal{O}_R)\).

\begin{lemma}
\label{lem:conormal-injection}
Let \(R\) be a reduced divisor supported on \(E\). The action on the conormal bundle of \(R\) induces a canonical injection
\[
 \ker(\mathcal{S}|_R \to \Theta_R) \hookrightarrow \mathcal{O}_R.
\]
\end{lemma}

\begin{proof}
A logarithmic vector field preserves the ideal \(\mathcal{I}_R\), and hence induces a derivation on \(\mathcal{O}_R\) and a differential operator on \(\mathcal{I}_R/\mathcal{I}_R^2\). Both depend only on its restriction as a section of \(\mathcal{S}|_R\). If the induced derivation on \(\mathcal{O}_R\) is zero, the action on \(\mathcal{I}_R/\mathcal{I}_R^2\) is \(\mathcal{O}_R\)-linear. The conormal sheaf is invertible because \(R\) is a Cartier divisor on the smooth surface \(V\). Its endomorphism sheaf is canonically \(\mathcal{O}_R\), so this action defines the asserted morphism. We check its injectivity locally.

At a smooth point of \(R\) that does not lie on \(E-R\), choose coordinates with \(R=(x=0)\) and write
\[
 \theta=ax\partial_x+b\partial_y.
\]
The induced derivation on \(\mathcal{O}_R\) is zero exactly when \(b \in (x)\). The conormal action is multiplication by \(a\bmod(x)\). If this function is zero, both coefficients of \(\theta|_R\) are zero.

At an intersection point of two components of \(R\), take \(R=(xy=0)\) and write
\[
 \theta=ax\partial_x+by\partial_y.
\]
In the ring \(\mathbb{C}\{x,y\}/(xy)\), zero induced derivation means \(a \in (y)\) and \(b \in (x)\). The conormal action is multiplication by
\[
 \frac{\theta(xy)}{xy}=a+b\pmod{(xy)}.
\]
If this element is zero, then \(a=-b\) belongs to \((x) \cap (y)=0\) in this ring. Thus both coefficients vanish, and \(\theta|_R=0\).

Finally, at an intersection point of \(R\) with \(E-R\), choose coordinates with \(R=(x=0)\) and the other component given by \((y=0)\). Write
\[
 \theta=ax\partial_x+by\partial_y.
\]
The induced derivation on \(\mathcal{O}_R\) is zero exactly when \(b \in (x)\), and the conormal action is multiplication by \(a\bmod(x)\). Its vanishing again gives \(\theta|_R=0\).

\end{proof}

\begin{lemma}
\label{lem:zero-injection}
There is a canonical injection
\[
 H^0(\mathcal{R},Q_{\mathcal{S}}) \hookrightarrow H^0(\mathcal{R},Q_{\mathcal{O}}).
\]
\end{lemma}

\begin{proof}
By Lemma~\ref{lem:zero-tangential-vanishing}, every section of \(Q_{\mathcal{S}}\) induces the zero \(Q_{\mathcal{O}}\)-valued derivation on \(\mathcal{O}_{\mathcal{R}}\), since \(\mathcal{R}\) is reduced. Tensoring the injection of Lemma~\ref{lem:conormal-injection} with \(Q_{\mathcal{O}}\) and taking global sections gives the result.

\end{proof}

\begin{proof}[Proof of Proposition~\ref{prop:zero-comparison}]
If \(\mathcal{R}=0\), then \(D^+=D^-\), and the assertion is immediate. We therefore assume that \(\mathcal{R} \ne 0\). Each connected component of \(\mathcal{R}\) has a boundary vertex, so \(H^0(Q_{\mathcal{L}})=0\): a section vanishes on the negative-degree boundary component and then, by matching, on the whole tree. Riemann--Roch on the reduced divisor \(\mathcal{R}\), whose connected components are trees of rational curves, and
\[
 \det Q_{\mathcal{S}} \simeq Q_{\mathcal{L}}\otimes Q_{\mathcal{O}}
\]
give
\[
 \Phi(D^+)-\Phi(D^-)=\dim I_{\mathcal{O}}-\dim I_{\mathcal{S}}.
\]
Lemma~\ref{lem:zero-vanishing} and the restriction exact sequence show that the map
\[
 H^0(\mathcal{O}_V(-D^-)) \to H^0(Q_{\mathcal{O}})
\]
is surjective. Lemma~\ref{lem:zero-injection} therefore gives
\[
 \dim I_{\mathcal{S}} \le h^0(Q_{\mathcal{S}})
 \le h^0(Q_{\mathcal{O}})=\dim I_{\mathcal{O}}.
\]
\end{proof}

\begin{proof}[Proof of Proposition~\ref{prop:cohomological-comparison}]
By Lemma~\ref{lem:nonzero}, the increment of \(\Phi\) is nonnegative at every step with nonzero residue. Proposition~\ref{prop:zero-comparison} shows that the sum of the increments at the zero-residue steps is also nonnegative. Summing the formula in Lemma~\ref{lem:increment} and using \eqref{eq:initial-Phi} and Lemma~\ref{lem:terminal} yields
$$
 h^1(V,\mathcal{L}(-E'))-h^1(V,\mathcal{S}(-E'))
 =\sum_{j=0}^{N-1}
 (c_{\mathcal{S}}(D_j,G_j)-c_{\mathcal{O}}(D_j,G_j)) \ge 0.
$$
\end{proof}

\subsection{Wahl's Rational Conjecture}
\label{subsec:rational-conjecture}

We now allow arbitrary minimal good resolution graphs of rational surface singularities.

\begin{proof}[Proof of Theorem~\ref{prop:rational-conjecture}]
If the resolution graph is linear, the assertion follows from Wahl~\cite[Proposition~4.9]{Wahl2015}. If it is star-shaped with one node, the assertion follows from~\cite[Proposition~4.8]{Wahl2015}.

Suppose that the graph has at least two nodes. Then \((X,0)\) is not quasi-homogeneous, since the minimal good resolution graph of a quasi-homogeneous normal surface singularity has at most one node~\cite[Theorem~(2.3.1)]{OrlikWagreich1971}. Equation~\eqref{eq:Wahl-restriction}, Proposition~\ref{prop:cohomological-comparison}, and \eqref{eq:twisted-L-comparison} give
\begin{equation*}
\begin{split}
h^1(V,\mathcal{S}) &=r-4+h^1(V,\mathcal{S}(-E'))\\
 & \le r-4+h^1(V,\mathcal{L}(-E')) \\
 &=h^1(V,\mathcal{L})-1\\
 &< h^1(V,\mathcal{L}).\qedhere
\end{split}
\end{equation*}
\end{proof}

\section{A cohomological bound from the resolution graph}
\label{sec:lattice-estimate}

We start this section by recalling the construction of the graph
classes \(\mathcal{A}\), \(\mathcal{B}\), and \(\mathcal{C}\), and
then prove the cohomological estimates needed for the smoothing
obstruction. Most of these arguments will be of combinatorial nature.

\subsection{Graph constructions}
\label{subsec:graph-constructions}

The classes \(\mathcal{A}\), \(\mathcal{B}\), and \(\mathcal{C}\) are defined by the following construction; see Stipsicz--Szab\'o--Wahl~\cite{StipsiczSzaboWahl2008}. For each class, start with the corresponding weighted tree below, denoted by \(\Gamma_0\) and called the \emph{initial graph}:
\begin{equation*}\label{eq:initial-graph}
\begin{aligned}
\begin{tikzpicture}[scale=0.75]
\node[bullet] (00) at (0,0) [label=below:\(-3\)] {};
\node[bullet] (10) at (1,0) [label=below:\(-1\)] {};
\node[bullet] (11) at (1,1) [label=above:\(-3\)] {};
\node[bullet] (20) at (2,0) [label=below:\(-3\)] {};
\node (01) at (-0.5,1) {\(\mathcal{A}\):};
\draw (00) -- (10);
\draw (10) -- (11);
\draw (10) -- (20);
\end{tikzpicture}
\quad
\begin{tikzpicture}[scale=0.75]
\node[bullet] (00) at (0,0) [label=below:\(-4\)] {};
\node[bullet] (10) at (1,0) [label=below:\(-1\)] {};
\node[bullet] (11) at (1,1) [label=above:\(-4\)] {};
\node[bullet] (20) at (2,0) [label=below:\(-2\)] {};
\node (01) at (-0.5,1) {\(\mathcal{B}\):};
\draw (00) -- (10);
\draw (10) -- (11);
\draw (10) -- (20);
\end{tikzpicture}
\quad
\begin{tikzpicture}[scale=0.75]
\node[bullet] (00) at (0,0) [label=below:\(-6\)] {};
\node[bullet] (10) at (1,0) [label=below:\(-1\)] {};
\node[bullet] (11) at (1,1) [label=above:\(-3\)] {};
\node[bullet] (20) at (2,0) [label=below:\(-2\)] {};
\node (01) at (-0.5,1) {\(\mathcal{C}\):};
\draw (00) -- (10);
\draw (10) -- (11);
\draw (10) -- (20);
\end{tikzpicture}
\end{aligned}
\end{equation*}

Apply a finite, possibly empty, sequence of the blow-up operations (B-1) and (B-2), followed by the modification (M).

\begin{enumerate}
\item[(B-1)] Blow up the unique vertex of weight \(-1\): decrease its weight to \(-2\) and attach a new vertex of weight \(-1\) to it by a single edge.

\item[(B-2)] Blow up an edge incident to the unique vertex of weight \(-1\): insert a new vertex of weight \(-1\) into the edge and decrease the weights of both original endpoints by one.

\item[(M)] Replace the weight of the unique vertex of weight \(-1\) by \(-4\), \(-3\), or \(-2\), according as the construction started from the graph corresponding to \(\mathcal{A}\), \(\mathcal{B}\), or \(\mathcal{C}\).
\end{enumerate}

For a graph \(\Gamma\) in one of these classes, fix such a construction.
Let \(m\) be the number of operations of type (B-1) in the chosen construction and let \(r\) be the number of leaves of the resulting graph \(\Gamma\).

For the cohomological estimates below, let \(f \colon (V,E) \to (X,0)\) be the minimal good resolution of a rational surface singularity with graph \(\Gamma\), and assume that \(\Gamma\) has at least two nodes. Write \(E=\sum_{i=0}^{n}E_i\) and \(E_i^2=-d_i\). We use \(\mathcal{S}\), \(\mathcal{L}\), and \(E'\) as in Section~\ref{sec:computation-sequences}.


For a weighted tree \(T\) with vertex weights \(-d_v\), and
$D\in \Lambda_T$ put 
\[
\Lambda_T=\bigoplus_v\mathbb{Z}E_v,\qquad
k_T(E_v)=d_v+2\val(E_v)-6,\qquad
\mathfrak{q}_T(D)=\frac{D^2+k_T(D)}2.
\]
Here \(k_T\) is the characteristic linear form determined by its values on the vertex basis.
If \(T\) is the resolution graph of \((V,E)\), then adjunction gives \(k_T=K_V-2c_1(\mathcal{L})\), hence \(\mathfrak{q}_T=q_{\mathcal{L}}\); cf. Némethi~\cite[\S6.3]{Nemethi2022}.

\begin{proposition}\label{prop:lattice-estimate}
Let \((X,0)\) be a rational surface singularity whose minimal good resolution graph \(\Gamma\) has at least two nodes and is of type \(\mathcal{A}\), \(\mathcal{B}\), or \(\mathcal{C}\). If the chosen construction of \(\Gamma\) contains \(m\) operations of type (B-1), then
\[
 h^1(V,\mathcal{L}) \le m.
\]
\end{proposition}

\begin{proof}
Let \(\Gamma_0\) be the initial graph of the chosen construction. Denote its central vertex by \(E_0\) and its three leaves by \(E_1,E_2,E_3\). We use the same symbols for the corresponding basis elements of \(\Lambda_{\Gamma_0}\). Their intersection numbers are
\[
 E_0^2=-1,\quad
 E_i^2=-d_i,\quad
 E_0E_i=1,\quad
 E_iE_j=0\quad (1 \le i \ne j \le 3),
\]
where the leaves are ordered so that
\[
 (d_1,d_2,d_3)=
 \begin{cases}
 (3,3,3),&\text{for type }\mathcal{A},\\
 (4,4,2),&\text{for type }\mathcal{B},\\
 (6,3,2),&\text{for type }\mathcal{C}.
 \end{cases}
\]

The intersection form of \(\Gamma_0\) is negative semidefinite,
since \(\sum_{i=1}^3d_i^{-1}=1\).
Write \(\xi_0=x_0E_0+\sum_{i=1}^3x_iE_i\). The following choices satisfy \(\xi_0E_v=k_{\Gamma_0}(E_v)\) on the vertex basis:
\[
\begin{array}{c|c|c|c}
 \text{Type}
 &(x_0;x_1,x_2,x_3)
 &(\xi_0E_0,\xi_0E_1,\xi_0E_2,\xi_0E_3)
 &\xi_0^2\\ \hline
 \mathcal{A}
 &(-1;0,0,0)
 &(1,-1,-1,-1)
 &-1\\
 \mathcal{B}
 &(0;0,0,1)
 &(1,0,0,-2)
 &-2\\
 \mathcal{C}
 &(0;-\frac{1}{3},\frac{1}{3},1)
 &(1,2,-1,-2)
 &-3
\end{array}
\]
Completing the square gives
\[
 \mathfrak{q}_{\Gamma_0}(D)=\frac{D^2+\xi_0D}{2}=\frac{1}{2}\left(D+\frac{\xi_0}{2}\right)^2-\frac{\xi_0^2}{8}.
\]
The first term on the right is nonpositive because the intersection form is negative semidefinite. The values of \(\xi_0^2\) in the table therefore imply
\[
 \mathfrak{q}_{\Gamma_0}(D)
 \le -\frac{\xi_0^2}{8}
 \le \frac{3}{8}
 <1.
\]

Since \(k_{\Gamma_0}\) is characteristic, \(\mathfrak{q}_{\Gamma_0}(D)\) is integral for \(D\in\Lambda_{\Gamma_0}\). Thus the bound \(\mathfrak{q}_{\Gamma_0}(D)<1\) gives \(\mathfrak{q}_{\Gamma_0}(D)\le0\).

We next examine how \(\mathfrak{q}\) changes under the blow-up operations. For either operation, let \(T\) be the graph before the blow-up and \(T_+\) the graph after it, and let \(\pi\colon T_+ \to T\) denote the corresponding blow-down. Write \(R\) for the basis element associated with the new exceptional vertex, and let \(C_\pi\) be the set of vertices corresponding to components through the blow-up center. Define the total-transform homomorphism \(\pi^*\colon\Lambda_T \to \Lambda_{T_+}\) by
\[
 \pi^*E_v=
 \begin{cases}
 E_v+R,&v \in C_\pi,\\
 E_v,&\text{otherwise}.
 \end{cases}
\]
Every integral cycle $D_+$ on \(T_+\) has a unique expression
\[
 D_+=\pi^*D+uR,\quad
 D \in \Lambda_T,\quad u \in \mathbb{Z}.
\]
For an operation of type (B-1),
\[
 \mathfrak{q}_{T_+}(D_+)-\mathfrak{q}_T(D)
 =-\frac{u(u+3)}{2}
 \le 1.
\]
For an operation of type (B-2),
\[
 \mathfrak{q}_{T_+}(D_+)-\mathfrak{q}_T(D)
 =-\frac{u(u+1)}{2}
 \le 0.
\]

Let \(T\) now denote the graph obtained after all the blow-up operations and immediately before the modification (M). Starting from the bound \(\mathfrak{q}_{\Gamma_0} \le 0\) and applying the preceding inequalities successively, we obtain
\[
 \mathfrak{q}_T(D) \le m
 \quad\text{for every }D \in \Lambda_T.
\]

The modification (M) preserves the vertices and edges and changes only the weight of the unique vertex of weight \(-1\). We therefore identify \(\Lambda_T\) and \(\Lambda_\Gamma\) as free abelian groups using their common vertex basis. Let \(a\) be the coefficient of this vertex in an integral cycle \(D\). Its weight changes from \(-1\) to \(-k\), where \(k=4\), \(3\), or \(2\) for types \(\mathcal{A}\), \(\mathcal{B}\), or \(\mathcal{C}\), respectively.  Hence
\[
 \mathfrak{q}_\Gamma(D) = \mathfrak{q}_T(D)-\frac{k-1}{2}a(a-1) \le \mathfrak{q}_T(D) \le m,
\]
where the first inequality follows from \(a(a-1) \ge 0\) for every integer \(a\).
Since \([\mathcal{L}]\ge E'\ge0\), \eqref{eq:cohomology-formula} gives
\[
 h^1(V,\mathcal{L})
 =\mathfrak{q}_\Gamma([\mathcal{L}])
 \le m. \qedhere
\]
\end{proof}


The preceding estimate and Proposition~\ref{prop:cohomological-comparison} give the following bound ruling out the existence of $\mathbb Q$HD smoothings for graphs with multiple nodes.

\begin{theorem}\label{thm:main-technical}
Let \((X,0)\) be a rational surface singularity whose minimal good resolution graph \(\Gamma\) is of type \(\mathcal{A}\), \(\mathcal{B}\), or \(\mathcal{C}\) and has at least two nodes. Then
\begin{equation*}
 h^1(V,\mathcal{S})+\sum_{i=0}^{n}(d_i-3) \le 0.
\end{equation*}
\end{theorem}

\begin{proof}
	Equation~\eqref{eq:Wahl-restriction},
	Proposition~\ref{prop:cohomological-comparison}, and
	\eqref{eq:twisted-L-comparison} give
	\[
		h^1(V,\mathcal{S})
		=r-4+h^1(V,\mathcal{S}(-E'))
		\le r-4+h^1(V,\mathcal{L}(-E'))
		=h^1(V,\mathcal{L})-1.
	\]
	The initial values of \(\sum_v(d_v-3)\) are \(-2\), \(-1\), and
	\(0\) in types \(\mathcal{A}\), \(\mathcal{B}\), and
	\(\mathcal{C}\), respectively. Each operation of type (B-1)
	decreases this sum by one, while each operation of type (B-2)
	preserves it. The modification (M) increases it by \(3\), \(2\),
	or \(1\), respectively. Thus
	\begin{equation}\label{eq:sum<=1-m}
		\sum_{v\in\Gamma}(d_v-3)=1-m.
	\end{equation}
	Combining this identity with Proposition~\ref{prop:lattice-estimate},
	we obtain
	\[
		h^1(V,\mathcal{S})+\sum_{v\in\Gamma}(d_v-3)
		\le h^1(V,\mathcal{L})-1+(1-m)
		=h^1(V,\mathcal{L})-m\le0.\qedhere
	\]
\end{proof}

\section{Rational homology disk smoothings}
\label{sec:qhd-smoothings}

\begin{proof}[Proof of Theorem~\ref{thm:main}]
Suppose that a normal surface singularity \((X,0)\) admits a \QHDS and
its minimal good resolution graph \(\Gamma\) has at least two
nodes. Then \((X,0)\) is rational, and the result of
Stipsicz--Szab\'o--Wahl recalled above implies that \(\Gamma\) is of
type \(\mathcal{A}\), \(\mathcal{B}\), or \(\mathcal{C}\).
By \eqref{eq:dimension} and Theorem~\ref{thm:main-technical},
the smoothing component would have nonpositive dimension,
contrary to the positive dimension of a smoothing component. 	
\end{proof}

\begin{proof}[Proof of Corollary~\ref{corollary:ClassificationOfGraphs}]
Let \((X,0)\) be a normal surface singularity admitting a \QHDS.
By Theorem~\ref{thm:main}, its minimal good resolution graph
has at most one node. If it has no nodes, it is the linear resolution graph of a cyclic quotient singularity. The assertion in this case follows from Looijenga--Wahl~\cite[(5.10)]{LooijengaWahl1986} and Wahl~\cite[(5.9.1)]{Wahl1981}. If it has one node, the graph is star-shaped, and the assertion follows from Bhupal--Stipsicz~\cite[Theorems~1.4 and 1.6]{BhupalStipsicz2011}.
\end{proof}

\begin{proof}[Proof of Corollary~\ref{cor:PreciseMain}]
  Corollary~\ref{corollary:ClassificationOfGraphs} gives the possible
  resolution graphs. The cyclic quotient graphs and the graphs of
  Figure~\ref{figure:fig1} are taut: each determines a unique analytic
  type, which is weighted
  homogeneous~\cite{Laufer1973,Pinkham1977}. The existence results
of \cite{BhupalStipsicz2011}
  for these graphs therefore give precisely the first two classes in the statement.

For a graph of Figure~\ref{figure:fig2}, the analytic type is determined by the four intersection points on the central curve together with their corresponding arms, and every such analytic type is weighted homogeneous; see Wahl~\cite[Introduction and proof of Corollary~8.2]{Wahl2011QHD}. Fowler's classification of model surfaces~\cite[Theorem~3.3.1 and Section~3.5]{Fowler2013}, summarized in Wahl~\cite[Corollary~1.5]{Wahl2021QHD}, determines exactly the cross ratios for which a \QHDS exists. The central curves in Fowler's configurations at infinity and in the resolution are naturally isomorphic, and their intersection points correspond to dual arms~\cite[Section~2.1.1]{Fowler2013}. Under the normalization \eqref{eq:cross-ratio}, Fowler's values are exactly those in the third item.

Only permutations induced by automorphisms of the weighted graph identify analytic types. For \((a,b,c;d)=(3,3,3;4)\), the three arms of weight \(-3\) may be permuted, so the two displayed values give the same analytic type. Thus each graph of Figure~\ref{figure:fig2} gives exactly one analytic type admitting a \QHDS.
\end{proof}

\begin{proof}[Proof of Corollary~\ref{cor:qhd-properties}]
By Corollary~\ref{cor:PreciseMain}, every normal surface singularity admitting a \QHDS is weighted homogeneous. Wahl~\cite[Corollary~8.2]{Wahl2011QHD} proves that every \QHDS component of such a singularity has dimension one. Moreover, every such smoothing is induced from a \(\mathbb{Q}\)-Gorenstein smoothing over the normalization of that component~\cite[Theorem~3.4]{Wahl-2011-log}. Pulling back the equivariant smoothing of the index-one cover preserves its quotient construction, so the induced smoothing is also \(\mathbb{Q}\)-Gorenstein.
\end{proof}

\end{document}